\documentclass{article}
\usepackage[utf8]{inputenc}
\usepackage{graphicx}
\usepackage{amsmath}
\usepackage{amssymb}
\usepackage{amscd}
\usepackage{tikz-cd}
\usepackage{authblk}
\usepackage{amsthm}
\newtheorem{theorem}{Theorem}

\newtheorem{corollary}{Corollary}[theorem]
\newtheorem{lemma}{Lemma}[theorem]
\title{The infinite Fibonacci cube and its generalizations}
\author{Hiep Trinh and Trevor M.\ Wilson }
\affil{Miami University}
\date{December 8, 2023}

\begin{document}

\maketitle

Please do not cite this version; it is a preliminary and very incomplete draft.

\begin{abstract}
The Fibonacci cube $\Gamma_n$ is is the graph whose vertices are independent subsets of the path graph of length $n$, where two such vertices are considered adjacent if they differ by the addition or removal of a single element. Klavžar \cite{Kla13} suggested considering the infinite Fibonacci cube $\Gamma_\infty$ whose vertices are independent subsets of the one-way infinite path graph with the same adjacency condition.
We show that every connected component of $\Gamma_\infty$
is asymmetric (has no nontrivial automorphism) and no two connected components of $\Gamma_\infty$ are isomorphic.
This follows from our results on a further generalization $\Gamma_G$ where $G$ is a simple, locally finite hypergraph with no isolated vertices.
\end{abstract}

\section{Definitions}

By a \emph{hypergraph} we will always mean a finitary hypergraph, which is a structure $G = (V,E)$ where each edge $e \in E$ is a finite subset of the vertex set $V$.
A hypergraph $G$ is \emph{simple} if every edge contains more than one vertex and no edge is a subset of any other edge.
A hypergraph $G$ is \emph{locally finite} if every vertex is contained in only finitely many edges.
Two vertices $a$ and $b$ are \emph{neighbors} if there is an edge containing both of them.
A set of vertices $A$ is a \emph{tail} of a vertex $b$ if $b \not \in A$ and $A \cup \{b\}$ is an edge.
A vertex is \emph{isolated} if it is not contained in any edge.
A set $s$ of vertices is \emph{independent} if no edge is a subset of $s$.
For two sets $s$ and $t$, their symmetric difference is denoted by $s\oplus t$.

For a hypergraph $G$, let $\Gamma_G$ be the graph whose vertices are the independent subsets of $G$, where two vertices $s$ and $t$ are adjacent iff they differ by the addition or removal of a single vertex: $\lvert s \oplus t \rvert = 1$.
For every vertex $s$ of $\Gamma_G$, let $\Gamma_G(s)$ denote the connected component of $\Gamma_G$ containing $s$.  Note that $\Gamma_G(s)$ consists of all the elements of $\Gamma_G$ (independent subsets of $G$) whose symmetric difference with $s$ is finite.

\section{Results}

\begin{theorem}
Let $S$ and $T$ be simple, locally finite hypergraphs,
and suppose there is an isomorphism $g$ from some connected component $\Gamma_S(s)$ of $\Gamma_S$ to some connected component $\Gamma_T(t)$ of $\Gamma_T$. Then,
\begin{enumerate}
\item $S$ is isomorphic to $T$, and
\item If $S$ and $T$ have no isolated vertices, there is an isomorphism $f : S \to T$ such that $g(r) = f[r]$ for every vertex $r$ of $\Gamma_S(s)$.
\end{enumerate}
\end{theorem}

In other words, part 2 of the theorem says that if $S$ and $T$ are simple, locally finite hypergraphs with no isolated vertices, then the only isomorphisms between connected components of the generalized Fibonacci cubes $\Gamma_S$ and $\Gamma_T$ are the ones induced by isomorphisms of the underlying hypergraphs $S$ and $T$.

We first prove part 1 of the theorem.
By redefining $t$ if necessary, we may assume without loss of generality that $f[s] = t$.

For every vertex $v$ of $S$, let $x_s(v)$ be the set of vertices  obtained by removing (if present) $v$ and every neighbor of $v$ from $s$. Let $y_s(v) = x_s(v) \cup \lbrace v \rbrace$. 
Both $x_s(v)$ and $y_s(v)$ are independent subsets of $S$ whose symmetric difference with $s$ is finite (since $S$ is locally finite,) and therefore both $x_s(v)$ and $y_s(v)$ are vertices of $\Gamma_S(s)$.
As $x_s(v)$ and $y_s(v)$ are adjacent in $\Gamma_S(s)$, it follows that $g(x_s(v))$ and $g(y_s(v))$ are adjacent in $\Gamma_T(t)$,
so their symmetric difference has only a single element which we may call $f(v)$.  In other words, $g(x_s(v)) \oplus g(y_s(v)) = \{f(v)\}$.

The following lemma shows that the definition of this function $f$ is not as particular as it might appear:

\begin{lemma}
For all $a,b \in \Gamma_S(s)$, if $a \oplus b = \lbrace v \rbrace$, then $g(a)\oplus g(b) = \lbrace f(v) \rbrace$.
\end{lemma}
\begin{proof}
Without loss of generality, assume that $b = a \cup \{v\}$. 
Let $x_0,x_1,x_2,...,x_{n+1}$ be a path in $\Gamma_S$ from the vertex $x_0 = x_s(v)$ to the vertex $x_{n+1} = a$
obtained by first removing elements of $x_s(v) \setminus a$ one by one (in some arbitrary order) and then adding elements of $a \setminus x_s(v)$ one by one (in some arbitrary order.)
Note that because neither of $x_s(v)$ and $a$ contains the vertex $v$, no $x_i$ can contain the vertex $v$.
Similarly because neither of $x_s(v)$ and $a$ contains a tail of $v$ (the latter because $a \cup \{v\} = b$ and $b$ is independent), no $x_i$ can contain a tail of $v$.
Therefore defining $y_i = x_i \cup \lbrace v \rbrace$, each $y_i$ is an independent set, and we obtain the following subgraph of $\Gamma_S(s)$.

\begin{equation*}
\begin{tikzcd}
y_s(v) = &y_0 \arrow[r, no head] & y_1 \arrow[r, no head, dotted] & y_n \arrow[r, no head] &y_{n+1} &= b \\
x_s(v) = &x_0 \arrow[u, no head] \arrow[r, no head] & x_1 \arrow[u, no head] \arrow[r, no head, dotted] & x_n \arrow[u, no head] \arrow[r, no head] & x_{n+1}  \arrow[u, no head] &=a
\end{tikzcd}
\end{equation*}

It remains to prove by induction that 
$g(x_i)\oplus g(y_i) = \lbrace f(v) \rbrace$ for all $i \le n+1$.
The case $i = 0$ follows from the definitions of $x_0$, $y_0$, and $f$.
Now assume that $i \le n$ and $g(x_i)\oplus g(y_i) = \lbrace f(v) \rbrace$.
Since $g$ is an isomorphism, it maps 4-cycles in $\Gamma_S(s)$ to 4-cycles in $\Gamma_T(t)$.
Consider the following 4-cycle in $\Gamma_T(t)$:
\begin{equation*}
\begin{tikzcd}
g(y_i) \arrow[r, no head]                   & g(y_{i+1})                    \\
g(x_i) \arrow[u, no head] \arrow[r, no head] & g(x_{i+1}) \arrow[u, no head]
\end{tikzcd}
\end{equation*}
Since $g(x_i)\oplus g(y_i) = \lbrace f(v) \rbrace$
neither $g(y_i) \oplus g(y_{i+1})$ nor $g(x_i) \oplus g(x_{i+1})$ can be $\lbrace f(v) \rbrace$, and it follows that $g(y_i) \oplus g(y_{i+1})$ and $g(i_n) \oplus g(x_{i+1})$ must be the same, and that $g(x_{i+1}) \oplus g(y_{i+1}) = \{f(v)\}$, as desired.  The case $i = n+1$ proves the lemma.
\end{proof}

\begin{lemma}
    $f$ is bijective.
\end{lemma}
\begin{proof}
First we show $f$ is injective.
Let $v$ and $v'$ be vertices of $S$ and suppose that $f(v) = f(v')$.
Define $x_s(v,v') = x_s(v) \cap x_s(v')$, which is the independent subset of $S$ obtained by removing $v$ and $v'$ and their (finitely many) neighbors from $s$.  Note that  $x_s(v,v')$ is a vertex of $\Gamma_S(s)$.
Then 
\[g(x_s(v,v') \cup \lbrace v \rbrace) = g(x_s(v,v')) \oplus f(v) = g(x_s(v,v') \oplus f(v') = g(x_s(a,b) \cup \lbrace v' \rbrace).\]
Since $g$ is injective, it follows that $v = v'$.

Now we show $f$ is surjective.
Let $w$ be a vertex of $T$.
Since $g$ is surjective, we have $x_t(w) = g(a)$ and $y_t(w) = g(b)$ for some vertices $a$ and $b$ of $\Gamma_S(s)$. Since $x_t(w)$ and $y_t(w)$ are adjacent and $g$ is an isomorphism, $a$ and $b$ are adjacent, so $a \oplus b = \{v\}$ for some vertex $v$ of $S$. Then by a previous lemma,
\[\{f(v)\} = g(a) \oplus g(b) = x_t(w) \oplus y_t(w) = \{w\},\] so $f(v) = w$.
\end{proof}

\begin{lemma}
For all $r \in \Gamma_S(s)$, we have $g(r) = f[r] \oplus c$, where $c = f[s] \oplus t$.

\end{lemma}
\begin{proof}
Note that the claim is true in the case $r = s$ because $g(s) = t$.
Consider a path from $s$ to $r$ 
in $\Gamma_S(s)$
obtained 
by first removing the elements of $s \setminus r$ one by one (in some arbitrary order) and then adding elements of $r \setminus s$ one by one (in some arbitrary order.) 
Since $s$ and $r$ are independent subsets of $S$, each set in this path is an independent subset of $S$,
and since any two consecutive sets in this path differ by a single element, this path is indeed a path in 
 $\Gamma_S(s)$.
    This lemma is proved by applying a previous lemma repeatedly on two consecutive independent sets on the path.
\end{proof}


\begin{lemma}
$f$ is an isomorphism from $S$ to $T$.
\end{lemma}
\begin{proof}
We will show that for every finite independent subset $a$ of $S$, the set $f[a]$ is an independent subset of $T$.
Then the same argument with $S$ and $T$ switched and $f$ and $g$ replaced by their inverses will show that for every finite independent subset $b$ of $T$, the set $f^{-1}[b]$ is an independent subset of $S$.  Putting these facts together, we have an equivalence: for every finite subset $a$ of $S$, it is independent if and only if the subset $f[a]$ of $T$ is independent.
Since the hypergraphs $S$ and $T$ are simple (and their edges are finite) it is not difficult to see that any bijection $f$ with this property must be an isomorphism.

Let $a$ be a finite independent subset of $S$.
Since $S$ is locally finite, there is
some $r \in \Gamma_S(s)$ such that $r$ does not contain any elements of $a$ or their neighbors.  Then
 for every subset $a'$ of $a$, the set
$r \oplus a'$ is independent in $S$, so $r \oplus a'$ is in $\Gamma_S(s)$,
and therefore by a previous lemma,
it follows that the set
$f[r] \oplus f[a'] \oplus c$ is in $\Gamma_T(t)$,
where $c = f[s] \oplus t$, so
$f[r] \oplus f[a'] \oplus c$ is independent in $T$.
Now letting
$a' = a - (f^{-1}[c] \oplus r)$,
we have
\[  f[r] \oplus f[a'] \oplus c = f[r] \oplus (f[a] - (c \oplus f[r])) \oplus c = 
f[a] \cup (c \oplus f[r]), \]
so the set $f[a] \cup (c \oplus f[r])$ is in independent in $T$.  Since $f[a]$ is a subset of This set, $f[a]$ must also be independent in $T$, as desired.

\end{proof}

Now we prove part 2 of the theorem.
Since a previous lemma gives a constant subset $c$ of $T$ such that we have 
\[\forall r \in \Gamma_S(s), \quad g(r) = f[r] \oplus c,\] it suffices to show 
under the additional assumption that $S$ and $T$ have no isolated vertices that $c = \emptyset$.  Suppose toward a contradiction that $c \ne \emptyset$.
Define an isomorphism 
\[h : \Gamma_S(s) \to \Gamma_S(f^{-1}[t])\]
by composing $g$ with the isomorphism $\Gamma_T(t) \to \Gamma_S(f^{-1}[t])$ induced by $f^{-1}$.  In other words,
$h(r) = f^{-1}[g(r)]$.
Then for all $r$ in $\Gamma_S(s)$.
\[h(r) = f^{-1}[g(r)]
= f^{-1}[f[r] \oplus c] = r \oplus f^{-1}[c].\]
Defining $b = f^{-1}[c]$, we have
\[\forall r \in \Gamma_S(s), \quad h(r) = r \oplus b,\] 
and by our assumption that $c \ne \emptyset$ it follows that $b \ne \emptyset$.
Since $h$ is an isomorphism from $\Gamma_S(s)$ to $\Gamma_S(f^{-1}[t])$, it follows that
for every independent subset $r$ of $S$ that differs finitely from $s$, the subset $r \oplus b$ of $S$ is also independent.

Since $b$ is nonempty, we may take a vertex $v \in b$.  Since $S$ has no isolated vertices, we have $v \in e$ for some edge $e$ of $S$.
Since $S$ is locally finite, there is some $s' \in \Gamma_S(s)$ such that $s'$ does not contain any elements of
$e$ or their neighbors.
Define $r = s' \cup e \setminus b$.
Then $r$ is in $\Gamma_S(s)$ and it follows that 
$h(r) = r \oplus b$ is in $\Gamma_S(f^{-1}[t])$,
so the set $r \oplus b$ is independent in $S$.
But this is a contradiction, because $e$ is an edge of $S$ and $e \subseteq r \oplus b$.











This concludes the proof of the theorem.

Since the infinite Fibonacci cube $\Gamma_\infty$ is $\Gamma_G$ where $G$ is the one-way infinite path graph, and $G$ itself is asymmetric, we obtain the following:

\begin{corollary}
Every connected component of $\Gamma_\infty$
is asymmetric, and no two connected components of $\Gamma_\infty$ are isomorphic.
\end{corollary}

\end{document}